\documentclass[leqno,12pt]{amsart}
\usepackage{latexsym}
\usepackage{mathrsfs}
\usepackage{amsmath}
\usepackage{amssymb}
\usepackage[bookmarks]{hyperref}
\usepackage{pstricks,pst-plot}

\newtheorem{theorem}{Theorem}[section]
\newtheorem{lemma}[theorem]{Lemma}

\theoremstyle{definition}

\newtheorem{definition}[theorem]{Definition}

\allowdisplaybreaks[1]

\newcommand{\R}{\mathbb{R}}

\newcommand{\oU}{\overline U}

\newcommand{\uc}{\mathscr U}
\newcommand{\rowun}[2]{#1_1\cup\ldots\cup#1_#2}
\newcommand{\row}[2]{#1_1,\ldots,#1_#2}
\newcommand{\range}[1]{1,\ldots, r_#1}

\begin{document}

\title[Continuous functions one-to-one almost everywhere]{Existence of continuous functions\\ that are one-to-one almost everywhere}
\author{Alexander J. Izzo}
\address{Department of Mathematics and Statistics, Bowling Green State University, Bowling Green, OH 43403}
\email{aizzo@math.bgsu.edu}



\begin{abstract}
\vskip 24pt
It is shown that given a metric space $X$ and a $\sigma$-finite positive regular Borel measure $\mu$ on $X$, there exists a bounded continuous real-valued function on $X$ that is one-to-one on the complement of a set of $\mu$ measure zero.
\end{abstract}
\maketitle

\vskip -1.36 true in
\centerline{\footnotesize\it Dedicated to the memory of Mary Ellen Rudin} 
\vskip 1.36 truein

\section{Introduction}

In \cite{IL} the author and Bo Li studied the question of how many functions are needed to generate an algebra dense in various $L^p$-spaces.  In connection with this, they proved \cite[Theorem~1.10]{IL} that on \hbox{every} smooth manifold-with-boundary there exists a bounded continuous {real-valued} function that is one-to-one on the complement of a set of \hbox{measure} zero.  It was suggested by Lee Stout that this result would generalize to a metric space context.  In this paper we show that this is indeed the case.   The author would like to thank Stout for sharing his insight. We state the result using the following terminology introduced in \cite{IL}.

\begin{definition}
We call a map $F$ defined on a measure space $X$ {\em one-to-one almost everywhere\/} if there is a subset $E$ of $X$ of measure zero such that the restriction of $F$ to $X\setminus E$ is one-to-one.
\end{definition}

\begin{theorem}\label{oneone}
Let $X$ be a metric space and $\mu$ be a $\sigma$-finite positive regular Borel measure on X.  Then there exists a bounded continuous real-valued function on X that is one-to-one almost everywhere.
\end{theorem}

The boundedness of the function is not really important; given an unbounded function with the other properties, we can obtain a bounded one by post composing with a homeomorphism of $\R$ onto the interval $(-1,1)$.  The point of the theorem is that the function is continuous {\em everywhere\/} and one-to-one {\em almost everywhere\/}.
Note that the metric space $X$ can be of arbitrarily large cardinality, but the set of full measure on which the function is one-to-one can have cardinality at most that of the continuum.  Note also that the theorem becomes false if the $\sigma$-finiteness condition is dropped as is exemplified by the case of counting measure on a discrete space with cardinality greater than that of the continuum.

The result about continuous one-to-one almost everywhere functions in \cite{IL} was used there to show that on every Riemannian manifold-with-boundary $M$ of finite volume there exists a bounded continuous real-valued function $f$ such that the set of polynomials in $f$ is dense in $L^p(M)$ for all $1\leq p<\infty$ \cite[Theorem~1.2]{IL}.  The argument given there can now be repeated using Theorem~\ref{oneone} above in place of \cite[Theorem~1.10]{IL} to establish the following more general result.  This result also strengthens \cite[Theorem~1.1]{IL}.

\begin{theorem}
Let $X$ be a metric space and $\mu$ be a finite positive regular Borel measure on X.  
Then there exists a bounded continuous real-valued function $f$ on X such that the set of polynomials in $f$ is dense in $L^p(\mu)$ for all $1\leq p<\infty$.
\end{theorem}

\section{Proof of Theorem~\ref{oneone}}

We begin with several lemmas.  The first of these is probably well-known, and it appears with proof as \cite[Lemma~3.1]{IL}.  Throughout the paper, by \lq\lq a Cantor set" we mean any space that is homeomorphic to the standard middle thirds Cantor set.

\begin{lemma}\label{lemma1}
If $C$ is a Cantor set and $\uc$ is an open cover of $C$, then $C$ can be written as a finite union $C=C_1\cup\ldots\cup C_N$ of disjoint Cantor sets $C_1,\ldots, C_N$ each of which lies in some member of $\uc$.
\end{lemma}

\begin{lemma}\label{topo}
Let $X$ be a topological space and $\mu$ be a $\sigma$-finite positive regular Borel measure on $X$.  Then there exists a countable collection $\{K_n\}$ of disjoint compact sets in $X$ such that $\mu\bigl(X\setminus (\bigcup K_n)\bigr)=0$.
\end{lemma}

\begin{proof}
By hypothesis $X=\bigcup_{n=1}^\infty X_n$ with $\mu(X_n)<\infty$ for each $n$, and without loss of generality the $X_n$ can be taken to be disjoint.  For each fixed $n$, the regularity of $\mu$ enables us to inductively choose disjoint compact sets $X_n^1, X_n^2, \ldots$ contained in $X_n$, such that 
$\mu\bigl(X_n\setminus (X_n^1\cup\ldots\cup X_n^j)\bigr)<1/j$ for each $j=1,2,\ldots\,$ .  Then $\mu\bigl(X_n\setminus (\bigcup_{j=1}^\infty X_n^j)\bigr)=0$.  Hence $\{X_n^j\}_{n,j}$ is a countable collection of disjoint compact sets in $X$ such that $\mu\bigl(X\setminus (\bigcup_{n,j} X_n^j)\bigr)=0$.
\end{proof}

\begin{lemma}\label{opensets}
Let $X$ be a (nonempty) compact metric space without isolated points, and let $\mu$ be a positive regular Borel measure on $X$.  Fix $\varepsilon>0$ and $\delta>0$.  Then for every sufficiently large positive integer $r$, there exists a collection $\{\row Ur\}$ of nonempty open sets in $X$ with disjoint closures such that
$$\mu\bigl(X\setminus (\rowun Ur)\bigr)<\varepsilon$$
and
$${\rm diameter}(U_j)<\delta \quad \hbox{for every\ } j=1,\ldots, r.$$
\end{lemma}

\begin{proof}
Since $X$ is a compact metric space, $X$ is totally bounded.  Thus $X$ can be covered by finitely many balls $\row As$ of diameters less than $\delta$.  Set $E_1=A_1$ and $E_j=A_j\setminus (\rowun A{{j-1}})$ \vadjust{\kern 2pt} for each $j=2,\ldots, s$.  Then the $E_j$ are disjoint and $\bigcup_{j=1}^s E_j=X$.  By the regularity of $\mu$, for each $j=1,\ldots, s$, we can choose a compact set $K_j$ contained in $E_j$ such that $\mu(E_j\setminus K_j)<\varepsilon/s$.  
Then the sets $\row Ks$ are disjoint and have diameters less than $\delta$.  Hence we can choose open neighborhoods $\row Us$ of $\row Ks$, respectively, so that the closures of the $U_j$ are disjoint and 
$${\rm diameter}(U_j)<\delta \quad \hbox{for every\ } j=1,\ldots, s.$$  
Then also
$$\mu\bigl(X\setminus (\rowun Us)\bigr)\leq \mu\bigl(X\setminus (\rowun Ks)\bigr)=
\sum_{j=1}^s \mu(E_j\setminus K_j)<\varepsilon.$$

The above argument establishes that the desired nonempty open sets can be obtained for {\em some\/} positive integer $r\leq s$.  To show that $r$ can be taken arbitrarily large, it suffices by induction, to show that $r$ can be increased by 1.  To this end, suppose that $\row Ur$ are as in the statement of the lemma.  Let $\gamma=\varepsilon-\mu\bigl(X\setminus (\rowun Ur)\bigr)>0$, and choose a point $p\in U_r$.  Because $X$ has no isolated points and $\mu$ is regular, there is a nonempty compact set $K$ in $U_r\setminus\{p\}$ such that $\mu\bigl((U_r\setminus\{p\})\setminus K\bigr)<\gamma$.  Choose open neighborhoods $U'_r$ and $U'_{r+1}$ of $\{p\}$ and $K$, respectively, contained in $U_r$ with disjoint closures.  Then $U_1,\ldots, U_{r-1}, U'_r, U'_{r+1}$ is a collection of $r+1$ nonempty open sets with the required properties.
\end{proof}

\begin{lemma}\label{largecantor}
Given a (nonempty) compact metric space $X$ without isolated points, a positive regular Borel measure $\mu$ on $X$, and $\varepsilon>0$, there exists a Cantor set $C$ in $X$ such that $\mu(X\setminus C)<\varepsilon$.
\end{lemma}

A result close to Lemma~\ref{largecantor} appears in the paper \cite{G} by Bernard Gelbaum.  (The author would like to thank Bo Li for pointing this out.)  Lemma~\ref{largecantor} is more general than the result in \cite{G}, since in \cite{G} the measure is required to be nonatomic and there is no such requirement in Lemma~\ref{largecantor}.  The author was surprised to find that the proof in \cite{G} is very different from the one given here.

\begin{proof}
By the preceding lemma, there are nonempty open sets $\row U{{r_1}}$ (for some $r_1$) with disjoint closures such that 
$$\mu\bigl(X\setminus(\rowun U{{r_1}})\bigr)<\varepsilon/2$$
and 
$${\rm diameter}(U_{j_1})<1 \quad \hbox{for every\ } j_1=1,\ldots, r_1.$$
Each $\oU_{j_1}$ is a compact set without isolated points, so we can apply the preceding lemma to each $\oU_{j_1}$ to obtain nonempty relatively open subsets $V_{{j_1},{j_2}}$ for $j_1=\range 1$ and $j_2=\range 2$ (for some $r_2$) with disjoint closures such that
\vadjust{\kern 6pt}
\item{\rm(i)} $V_{{j_1},{j_2}}\subset \oU_{j_1}$,
\item{\rm(ii)} $\mu\bigl(\oU_{j_1} \setminus (V_{{j_1},1}\cup\ldots\cup V_{{j_1},{r_2}})\bigr)< \displaystyle\frac{\varepsilon}{2^2\, r_1}$, and
\item{\rm(iii)} ${\rm diameter}(V_{{j_1},{j_2}})<1/2$.
\vadjust{\kern 6pt}\hfil\break
Setting $U_{{j_1},{j_2}}=V_{{j_1},{j_2}}\cap U_{j_1}$, we obtain nonempty open subsets of $X$ such that
\vadjust{\kern 6pt}
\item{\rm(i$'$)} $U_{{j_1},{j_2}}\subset U_{j_1}$,
\item{\rm(ii$'$)} $\mu\bigl(U_{j_1} \setminus (U_{{j_1},1}\cup\ldots\cup U_{{j_1},{r_2}})\bigr)< \displaystyle\frac{\varepsilon}{2^2\, r_1}$, and
\item{\rm(iii$'$)} ${\rm diameter}(U_{{j_1},{j_2}})<1/2$.
\vadjust{\kern 6pt}\hfil\break
In general, assume that we have chosen, for each $s=1,\ldots, k$, nonempty open subsets $U_{\row js}$ of $X$ for each $j_1=\range 1;\ldots; j_s=\range s$ (for some $\row rs$) with disjoint closures such that
\vadjust{\kern 6pt}
\item{\rm(i$''$)} $U_{\row js}\subset U_{\row j{{s-1}}}$,
\item{\rm(ii$''$)} $\mu\bigl(U_{\row j{{s-1}}} \setminus (U_{\row j{{s-1}},1}
\cup\ldots\cup U_{\row j{{s-1}}, r_s})\bigr)< \displaystyle\frac{\varepsilon}{2^s\, r_{s-1}}$, and
\item{\rm(iii$''$)} ${\rm diameter}(U_{\row js})<1/s$.
\vadjust{\kern 6pt}\hfil\break
Each $\oU_{\row jk}$ is a compact set without isolated points to which we can apply the procedure above to obtain open sets $U_{\row j{{k+1}}}$ for each $j_1=\range 1;\ldots; j_{k+1}=\range {{k+1}}$ (for some $r_{k+1}$) with disjoint closures such that conditions (i$''$)--(iii$''$) hold with $s$ replaced by $k+1$.  Thus by induction the construction can be continued.  

Now consider the sets $K_s=\bigcup_{{j_1}=1}^{r_1}\cdots\bigcup_{{j_s}=1}^{r_s} \oU_{\row js}$.  These are\break nonempty compact sets such that $K_1\supset K_2\supset \cdots$, so their intersection $C=\bigcap_{s=1}^\infty K_s$ is nonempty.  Moreover, one easily verifies that $\mu(X\setminus C)<\varepsilon$.  Finally we claim that $C$ is a Cantor set.  To verify this, note that for each sequence $(j_1, j_2, \ldots)\in \prod_{k=1}^\infty \{\range k\}$ we have
$$\oU_{j_1}\supset \oU_{{j_1},{j_2}}\supset \oU_{{j_1},{j_2},{j_3}}\supset\cdots,$$
so the intersection of these sets is nonempty, and because the diameters of these sets go to zero, the intersection consists of a single point.  Thus there is a well-defined map
$$\textstyle F: \prod\limits_{k=1}^\infty \{\range k\} \rightarrow C$$
sending the sequence $(j_1, j_2, \ldots)$ to the point in the intersection.  One easily  verifies that $F$ is a bijection by using that, for each fixed $s$, the sets $\oU_{\row js}$ (as $\row js$ vary) are disjoint.  One easily verifies that $F$ is continuous using that the diameters of the sets $\oU_{\row js}$ go to zero as $s\rightarrow \infty$.  Hence, by compactness, $F$ is a homeomorphism.  Thus since $\prod_{k=1}^\infty \{\range k\}$ is a Cantor set, so is $C$.
\end{proof}

\begin{lemma}\label{ncantor}
Given a (nonempty) compact metric space $X$ without isolated points and a positive regular Borel measure $\mu$ on $X$, there exists an at most countable collection $\{C_n\}$ of disjoint Cantor sets in $X$ such that $\mu\bigl(X\setminus(\bigcup C_n)\bigr)=0$.
\end{lemma}

\begin{proof}
We construct the sets $C_n$ inductively.  By the preceding lemma, there exists a Cantor set $C_1$ in $X$ such that $\mu(X\setminus C_1)<1$.  In general, assume that disjoint Cantor sets $\row Ck$ have been chosen such that $\mu\bigl(X\setminus (\rowun Ck)\bigr)<1/2^k$.  If in fact $\mu\bigl(X\setminus (\rowun Ck)\bigr)=0$, then we are done.  Otherwise, by the regularity of $\mu$, there is an open neighborhood $U\subsetneq X$ of $\rowun Ck$ such that $\mu\bigl(U\setminus (\rowun Ck)\bigr)<1/2^{k+2}$.  Now choose an open neighborhood $V$ of $\rowun Ck$ such that $\overline V\subset U$.  Let $Y=\overline{X\setminus \overline V}$.  Then $Y$ is a nonempty compact set disjoint from $\rowun Ck$ and $X=U\cup Y$.  Because $Y$ is the closure of the open set $X\setminus \overline V$, we see that $Y$ has no isolated points.  Therefore, the preceding lemma gives that there is a Cantor set $C_{k+1}$ in $Y$ such that $\mu(Y\setminus C_{k+1})<1/2^{k+2}$.  Since $C_{k+1}\subset Y$, we know that $C_{k+1}$ is disjoint from the sets $\row Ck$.  Since $X=U\cup Y$ we have
\begin{eqnarray*}
\mu\bigl(X\setminus (\rowun C{{k+1}})\bigr)&\leq &
\mu\bigl(U\setminus (\rowun Ck)\bigr) + \mu(Y\setminus C_{k+1}) \\
& < & 1/2^{k+2} + 1/2^{k+2} = 1/2^{k+1}.
\end{eqnarray*}
Thus by induction we obtain a sequence of disjoint Cantor sets $C_1,C_2,\ldots$, such that $\mu\bigl(X\setminus (\rowun Cj)\bigr)<1/2^j$ for every $j$.  Hence $\mu(X\setminus \bigcup_{n=1}^\infty C_n)=0$.
\end{proof}

\begin{lemma}\label{lemma3}
Given a metric space $X$ and a $\sigma$-finite positive regular Borel measure $\mu$ on $X$, there exist an at most countable collection $\{C_n\}$ of disjoint Cantor sets in $X$ and an at most countable set $S$ in $X$ disjoint from each $C_n$ such that $\mu\bigl(X \setminus \bigl((\bigcup C_n)\cup S\bigl)\bigl)=0$.
\end{lemma}

\begin{proof}
By Lemma~\ref{topo} there exists a countable collection $\{K_n\}$ of disjoint compact sets in $X$ such that $\mu\bigl(X\setminus (\bigcup K_n)\bigr)=0$.  By the 
Cantor-Bendixson theorem \cite[Theorem~2A.1]{Mos}, each of the compact sets $K_n$ is a disjoint union of a perfect set $P_n$ and an at most countable set $S_n$.  By Lemma~\ref{ncantor} each nonempty perfect set $P_n$ contains an at most countable collection $\{K_n^j\}_j$ of disjoint Cantor sets such that $\mu(P_n\setminus \bigl(\bigcup_j K_n^j)\bigr)=0$.  Now $\{K_n^j\}_{n,j}$ is an at most countable collection of disjoint Cantor sets, the set $S=\bigcup S_n$ is at most countable and disjoint from each $K_n^j$, and $\mu\bigl(X\setminus ((\bigcup_{n,j} K_n^j)\cup S)\bigr)=0$.
\end{proof}

With these preliminaries, we can now prove Theorem~\ref{oneone} by essentially repeating the proof of \cite[Theorem~1.10]{IL}.  Minor changes are required on account of the (possible) presence of the at most countable set $S$ in Lemma~\ref{lemma3}.  The proof will be carried out as if the collection $\{C_n\}$ and the set $S$ in Lemma~\ref{lemma3} are both countably infinite.  If either is actually finite, then in the inductive procedure below one simply ceases to carry out the part of the construction that no longer makes sense once the collection  $\{C_n\}$, or the set $S$, has been exhausted.  If both the collection $\{C_n\}$ and the set $S$ are finite, then the procedure terminates, but in that case the result is rather trivial, so the construction below is not really needed then.

\begin{proof}
[Proof of Theorem~\ref{oneone}]
By Lemma~\ref{lemma3} there exist in $X$ disjoint sets $S$ and $C_1, C_2, \ldots$ such that $S$ is at most countable, each $C_j$ is a Cantor set, and $\mu\bigl(X \setminus \bigl((\bigcup C_j)\cup S\bigl)\bigl)=0$.  Let the points of $S$ be denoted by $x_1, x_2, \ldots\,$ .   We will construct a sequence $(f_n)_{n=1}^\infty$ of continuous functions from $X$ into $[0,1]$ such that for each $n$
\item{\rm(i)} $f_n$ is one-to-one on $\rowun{C}{n}\cup \{x_1,\ldots, x_n\}$,
\item{\rm(ii)} $f_{n+1}$ agrees with $f_n$ on $\rowun{C}{n}\cup \{x_1,\ldots, x_n\}$, and
\item{\rm(iii)} $\|f_{n+1}-f_{n}\|_\infty\leq 1/2^n$.\hfil\break
Suppose for the moment that such a sequence of functions has been constructed.  Then on account of condition (iii), the sequence $(f_n)$ converges uniformly to a continuous limit function $f$.  Due to condition (ii), $f_m$ agrees with $f_n$ on $\rowun{C}{n}\cup \{x_1,\ldots, x_n\}$ for all $m\geq n$, and hence the limit function $f$ also agrees with $f_n$ on $\rowun{C}{n}\cup \{x_1,\ldots, x_n\}$.  Now given distinct points $a$ and $b$ in $(\bigcup_{j=1}^\infty C_j)\cup S$, choose $N$ such that both $a$ and $b$ lie in $\rowun CN\cup \{x_1,\ldots, x_N\}$.  Then $f(a)=f_N(a)\neq f_N(b)=f(b)$.  Hence $f$ is one-to-one on $(\bigcup C_j)\cup S$.
Thus it suffices to construct a sequence of functions satisfying conditions (i)--(iii).

We will construct the sequence of functions $f_n$ by induction.  For the purpose of carrying out the induction we will also require the additional condition that for each $n$
\item{\rm(iv)} $\{\, f_n(C_1),\ldots, f_n(C_n)\, \}$ is a collection of disjoint Cantor sets in $[0,1]$.

We begin by defining $f_1$.  Choose a Cantor set $\widetilde C_1$ in $[0,1]$ and a point $y_1$ in $[0,1]\setminus \widetilde C_1$.  Choose a homeomorphism $g_1$ of $C_1$ onto $\widetilde C_1$.  By the Tietze extension theorem, there is an extension of $g_1$ to a continuous function of $X$ into $[0,1]$ that maps $x_1$ to $y_1$.  Let $f_1$ be the extension.

Now to carry out the induction, assume that functions $\row fk$ have been defined so that conditions (i)--(iv) hold for those values of $n$ for which they are meaningful.  We wish to define $f_{k+1}$.  By the continuity of $f_k$, there is an 
open cover $\uc$ of $C_{k+1}$ such that for each member $U$ of $\uc$ we 
have that $f_k(U)$ is contained in an interval of length $1/2^k$.   By Lemma~\ref{lemma1} we can write $C_{k+1}$ as a finite union $C_{k+1}=C_{k+1}^1\cup\ldots\cup C_{k+1}^N$ of disjoint Cantor sets $C_{k+1}^1,\ldots, C_{k+1}^N$ each of which is contained in some member of $\uc$.  Then for each  $j=1,\ldots, N$, the set $f_k(C_{k+1}^j)$ is contained in an interval $I_{k+1}^j\subset [0,1]$ of length $1/2^k$.  Since $f_k(C_1),\ldots, f_k(C_k)$ are disjoint Cantor sets, their union is also a Cantor set and in particular has empty interior in $[0,1]$.  Consequently, we can choose disjoint Cantor sets 
$\widetilde C_{k+1}^1,\ldots, \widetilde C_{k+1}^N$  with $\widetilde C_{k+1}^j$ contained in $I_{k+1}^j\setminus \bigl(f_k(C_1)\cup\ldots\cup f_k(C_k)\cup \{f_k(x_1),\ldots, f_k(x_k)\}\bigr)$ for each $j$, and we can choose a point $y_{k+1}$ in $[0,1]\setminus \bigl(f_k(C_1)\cup\ldots\cup f_k(C_k)\cup \{f_k(x_1),\ldots, f_k(x_k)\}\cup \widetilde C_{k+1}^1\cup\ldots\cup \widetilde C_{k+1}^N\bigr)$ with $|f_k(x_{k+1})-y_{k+1}|<1/2^k$.  Choose a homeomorphism $g_{k+1}^j$ of $C_{k+1}^j$ onto $\widetilde C_{k+1}^j$ for each $j$, and then define $g_{k+1}$ on $\rowun C{{k+1}}  \cup \{\row x{{k+1}}\}$ by
$$g_{k+1}(x) =  
\begin{cases} f_k(x)  &  {\rm if \ \ } x \in \rowun Ck\cup\{\row xk\} \cr 
g_{k+1}^j(x)  & {\rm if \ \ } x \in C_{k+1}^j\qquad (j=1,\ldots, N)\cr
y_{k+1} & {\rm if \ \ } x=x_{k+1}
\end{cases} $$ 
Then $g_{k+1}$ is a homeomorphism of $\rowun C{{k+1}}\cup \{\row x{{k+1}}\}$ onto 
$f(C_1)\cup\ldots\cup f(C_k) \cup \widetilde C_{k+1}^1\cup\ldots\cup \widetilde C_{k+1}^N\cup \{\row y{{k+1}}\}$ taking $C_{k+1}$ onto $\widetilde C_{k+1}^1\cup\ldots\cup \widetilde C_{k+1}^N$.  Note that 
\smallskip
\[
\sup
\bigl\{|f_k(x)-g_{k+1}(x)|: {x\in {\rowun C{{k+1}}\cup\{\row x{{k+1}}\}}}\bigr\}\leq 1/2^k
\smallskip
\]
since for each $j$ both $f_k(C_{k+1}^j)$ and $g_{k+1}(C_{k+1}^j)$ are contained in the interval $I_{k+1}^j$ of length $1/2^k$ and $|f_k(x_{k+1})-y_{k+1}|<1/2^k$.  By the Tietze extension theorem, there is a continuous function $h_{k+1}$ on $X$ that agrees with $f_k-g_{k+1}$ on $\rowun C{{k+1}}\cup \{\row x{{k+1}}\}$ and satisfies
\[ \|h_{k+1}\|_{\infty}\leq 1/2^k.\]
Define a function $f_{k+1}$ on $X$ by
$$f_{k+1}(x) =  
\begin{cases} f_k(x)-h_{k+1}(x)  &  {\rm if \ \ } 0\leq f_k(x)-h_{k+1}(x) \leq 1\cr 
0  & {\rm if \ \ }  f_k(x)-h_{k+1}(x)\leq 0\cr
1 & {\rm if \ \ }  f_k(x)-h_{k+1}(x)\geq 1
\end{cases} $$ 
Then $f_{k+1}$ is a continuous functions from $X$ into $[0,1]$ such that
$f_{k+1}=g_{k+1}$ on $\rowun C{{k+1}}\cup \{\row x{{k+1}}\}$ and $\|f_{k+1}-f_k\|_\infty\leq 1/2^k$.  It follows that $\row f{{k+1}}$ satisfy the required conditions (i)--(iv) for those values of $n$ for which the conditions are meaningful.  Therefore, by induction we obtain the desired sequence $(f_n)$, and the proof is complete.
\end{proof}

{\bf Acknowledgement\/}: The author thanks the referee for reading the paper especially carefully and making several helpful comments that led to improvements.


\begin{thebibliography}{BC87}

\bibitem{G} B. R. Gelbaum,  {\it Cantor sets in metric measure spaces\/},
Proc.\ Amer.\ Math.\ Soc.\ {\bf 24\/} (1970), 341--343. 

\bibitem{IL}  A. J. Izzo and B. Li,  {\it Generators for algebras dense in $L^p$-spaces\/}, Studia Math.\  {\bf 217\/} (2013), 243-263.

\bibitem{Mos} Y. N. Moschovakis,  {\it Descriptive Set Theory\/}, Studies in Logic and the Foundations of Mathematics, {\bf 100\/}, North-Holland Publishing Co., Amsterdam-New York, 1980.
 
\end{thebibliography}
\end{document}